\documentclass[a4paper,11pt]{amsart}

\usepackage{mathpazo}
\usepackage{amssymb}
\usepackage[pagebackref,
    ,pdfborder={0 0 0}
    ,urlcolor=black,a4paper,hypertexnames=false]{hyperref}
\hypersetup{pdfauthor=Clara L\"oh,pdftitle=A note on bounded-cohomological dimension of discrete groups}

\usepackage{pgf}
\usepackage{tikz}
\usetikzlibrary{decorations.pathmorphing,calc}
\usepackage[all]{xy}
\SelectTips{cm}{}
\usepackage{pifont}

\newtheorem{thm}{Theorem}[section]
\newtheorem{prop}[thm]{Proposition}
\newtheorem{lem}[thm]{Lemma}
\newtheorem{cor}[thm]{Corollary}

\newtheorem{question}[thm]{Question}

\theoremstyle{remark}
\newtheorem{rem}[thm]{Remark}
\newtheorem{exa}[thm]{Example}

\theoremstyle{definition}
\newtheorem{defi}[thm]{Definition}

\usepackage{amsfonts}
\newcommand{\Z}{\mathbb{Z}}

\newcommand{\R}{\mathbb{R}}
\newcommand{\N}{\mathbb{N}}

\DeclareMathOperator*{\bigfree}{\bigstar}
\DeclareMathOperator{\id}{id}

\def\epsilon{\varepsilon}

\DeclareMathOperator{\im}{im}
\DeclareMathOperator{\bcd}{bcd}
\DeclareMathOperator{\cdim}{cd}
\newcommand{\cd}{\cdim_\Z}

\DeclareMathOperator{\Grp}{{\sf Group}}

\DeclareMathOperator{\Homeo}{Homeo}

\def\HK#1{%
  \Homeo_K(\R^{#1})}

\def\args{\;\cdot\;}

\def\rlhom#1#2{%
  {\overline{H}^{\ell^1}_{#1}}(#2; \R)}
\def\rbc#1#2{%
  \smash{\overline{H}_b^{#1}}(#2;\R)}

\makeatletter
\newcommand\norm{\bBigg@{0.8}}

\makeatother

\newcommand{\inparens}[2][norm]{\csname #1l\endcsname(#2%
                                 \csname #1r\endcsname)\mathclose{}}

\newcommand{\lob}[3][norm]{%
  {\overline{b}^{\ell^1}_{#2}\inparens[#1]{#3}}}

\usepackage{color}
\usepackage{pdfcolmk}

\long\def\forget#1{}

\def\draftinfo{}

\author{Clara L\"oh}
\title[Bounded-cohomological dimension of groups]{A note on bounded-cohomological dimension\\ of discrete groups}
\date{\today.\ 
    This work was supported by the CRC~1085 \emph{Higher Invariants} 
    (Universit\"at Regensburg, funded by the DFG).
    \draftinfo\\
     MSC~2010 classification: 55N35, 20J06, 20E99}
\begin{document}

\begin{abstract}
  Bounded-cohomological dimension of groups is a relative of classical
  cohomological dimension, defined in terms of bounded cohomology with
  trivial coefficients instead of ordinary group cohomology. We will
  discuss constructions that lead to groups with infinite
  bounded-cohomological dimension, and we will provide new examples of
  groups with bounded-cohomological dimension equal to~$0$. In
  particular, we will prove that every group functorially embeds into
  an acyclic group with trivial bounded cohomology.
\end{abstract}

\maketitle

\section{Introduction}

Bounded cohomology~$H_b^*(\args;\R)$ is a functional-analytic version
of ordinary group cohomology, defined in terms of cocycles that are
bounded with respect to the $\ell^1$-norm on the bar
complex~\cite{vbc,ivanov,monod,buehler}
(Section~\ref{sec:bcl1}). Bounded cohomology has various applications
in geometry and geometric group
theory~\cite{vbc,loeh,monod,monodicm}. There is a natural comparison
map between bounded cohomology and ordinary group cohomology with
$\R$-coefficients; however, this comparison map in general is neither
surjective nor injective, and bounded cohomology usually is hard to
calculate.

We will consider the following bounded analogue of classical
cohomological dimension of groups with trivial coefficients (which
should not be confused with the bounded-cohomological dimension with
varying coefficients~\cite{monodicm}):

\begin{defi}[bounded-cohomological dimension~\protect{\cite{ho}}]
  The \emph{bounded-co\-ho\-mological dimension} of a group~$G$ is defined by 
  \[
  \bcd (G) := \sup \bigl\{ n \in \N \bigm| H^n_b(G;\R) \not\cong 0 \bigr\}
  \in \N \cup \{\infty\}.
  \]
\end{defi}

In contrast with the corresponding invariant for ordinary group
cohomology, not much is known about bounded-cohomological
dimension. For example, bounded-co\-ho\-mological dimension does
\emph{not} admit an obvious bound in terms of the geometric dimension
of groups.

In this article, we will provide new examples of groups with
bounded-cohomological dimension equal to~$0$ as well as of basic
constructions that lead to groups with infinite bounded-cohomological
dimension.

For all amenable groups~$G$ one has~$\bcd(G) =
0$~\cite{vbc,ivanov}. For all groups~$G$ we have~$H^1_b(G;\R) \cong
0$~\cite{mitsumatsu} and hence $\bcd(G) \neq 1$. Free groups~$F$ of
rank at least~$2$ satisfy~$\bcd(F) \geq
3$~\cite{somasurf,somanonbanach,yoshida}; however, the exact value
of~$\bcd(F)$ is unknown. If $M$ is an oriented closed connected
$n$-manifold with non-zero simplicial volume, then~$\bcd \pi_1(M) \geq
n$~\cite{vbc}; this happens, for example, if $M$ admits a metric of
negative sectional curvature~\cite{inoueyano}. More generally, if $G$
is a hyperbolic group, then the comparison map~$H^*_b(G;\R)
\longrightarrow H^*(G;\R)$ is surjective in degree at
least~$2$~\cite{mineyev}, which gives lower bounds on~$\bcd
G$. Bounded cohomology in degree~$2$ is rather well understood in
terms of quasi-morphisms/pseudo-characters~\cite{grigorchuk}. For
example, $\bcd G \geq 2$ whenever $G$ is a sufficiently non-trivial
amalgamated free product~\cite{grigorchuk,grigorchuk2,fujiwara}.

No examples of groups~$G$ with~$\bcd(G) \not \in \{0,\infty\}$ seem to
be known.

\subsection*{Groups with small bounded cohomology}

Mather~\cite{mather} showed that the (discrete) group~$\HK n$ of
homeomorphisms~$\R^n \longrightarrow \R^n$ with compact support is
\emph{acyclic} for all~$n \in \N_{> 0}$, i.e., $H_k(\HK n;\Z) \cong 0$
for all~$k\in \N_{>0}$. Matsumoto and Morita~\cite{mm} refined Mather's
proof in the normed setting to obtain~$\bcd \HK n = 0$.  This was the
first example of a non-amenable group with trivial
bounded-cohomological dimension.

Baumslag, Dyer, and Heller~\cite[Section~4]{bdh} considered so-called
mitotic groups (see Section~\ref{subsec:mitoticgroups} for the
definition); mitotic groups have all the algebraic properties
necessary to carry out Mather's argument and Baumslag, Dyer,
Heller~\cite[Theorem~4.2]{bdh} proved that all mitotic groups are
acyclic.

Based on the normed refinement of Matsumoto and Morita of Mather's
proof, we will adapt the argument of Baumslag, Dyer, Heller to show
that mitotic groups have trivial bounded cohomology
(Section~\ref{sec:small}):

\begin{thm}[bounded cohomology of mitotic groups]\label{thm:small}
  If $G$ is a mitotic group, then $\bcd G = 0$.
\end{thm}

\begin{cor}[embedding groups into very acyclic groups]
  There is a functor~$M \colon \Grp \longrightarrow \Grp$ 
  and a natural transformation~$i \colon \id_{\Grp} \Longrightarrow M$ 
  with the following properties:
  \begin{enumerate}
    \item For all groups~$G$ the group~$M(G)$ is mitotic; in
      particular, $M(G)$ is acyclic and~$\bcd M(G) = 0$.
    \item For all groups~$G$, the homomorphism~$i_G \colon G
      \longrightarrow M(G)$ is injective.
    \item If $G$ is an infinite group, then $|M(G)| = |G|$, where $|\cdot|$ 
      denotes the cardinality.
  \end{enumerate}
\end{cor}
\begin{proof}
  Baumslag, Dyer, Heller~\cite[Section~5, Theorem~4.2]{bdh}
  constructed a functor~$M$ with these properties;
  Theorem~\ref{thm:small} is only needed to deduce that $\bcd M(G) =
  0$ for all groups~$G$.
\end{proof}

In particular, mitotic groups in general are \emph{not} amenable: For
instance, $M(F_2)$ contains the non-amenable group~$F_2$ as subgroup.
Moreover, all algebraically closed groups are
mitotic~\cite[Corollary~4.4]{bdh}.

\begin{exa}
  Clearly, not every group~$G$ with~$\bcd G = 0$ is acyclic:
  For every~$n \in \N \cup \{\infty\}$ there is a group~$G$ that is 
  not acyclic and satisfies 
  \[ \bcd G = 0 
     \quad \text{and} 
     \quad \cd G = n = \cdim_\R G,
  \] 
  e.g., one can consider the amenable group~$G = \Z^{\oplus n}$.
\end{exa}

\subsection*{Groups with large bounded cohomology}

On the other hand, it is not hard to construct groups with large
bounded cohomology, and hence of infinite bounded-cohomological
dimension. For example, even though there does not seem to be a
general K\"unneth theorem for bounded cohomology, we can use the
interplay between bounded cohomology and $\ell^1$-homology and 
(co)homological cross-products to propagate non-trivial classes:

\begin{prop}\label{prop:large}
  For each~$n \in \N$ let $G_n$ be a group with~$H^2_b(G_n;\R) \not \cong 0$, 
  and let $G \in \{ \bigoplus_{n \in \N} G_n, \prod_{n \in \N} G_n \}$. 
  Then
  \[ \bcd G = \infty.
  \]
  More precisely: There exists a sequence~$(\varphi_n)_{n \in \N}
  \subset H^2_b(G;\R)$ such that for all~$n \in \N$ we have
  \[ \varphi_0 \cup \dots \cup \varphi_{n-1} \neq 0 \in H_b^{2 \cdot n}(G;\R). 
  \]
\end{prop}

Here, $\bigoplus_{n \in \N} G_n$ denotes the subgroup of~$\prod_{n \in \N} G_n$ 
of families with finite support.

The proof of Proposition~\ref{prop:large} is given in
Section~\ref{subsec:examples}, where we also give further classes of
examples whose bounded cohomology can be easily calculated to a large
extent.

\begin{exa}
  Let $G := \bigoplus_\N F_2$. Then $G$ clearly is not acyclic and
  because of~$H_b^2(F_2;\R) \not\cong 0$ and $H^1(F_2;\R) \not\cong 0$
  we obtain
  \[ \bcd G = \infty
     \quad\text{and}\quad
     \cd G = \infty = \cdim_\R G.
  \]
\end{exa}

\begin{exa}
  There are acyclic groups with infinite bounded-cohomological
  dimension: For example, we can consider Higman's group
  \[  H := \langle a,b,c,d 
           \mid b^{-1}ab = a^2, c^{-1}bc = b^2, d^{-1}cd = c^2, a^{-1}da = d^2 
           \rangle;
  \]
  it is well known that $H$ is acyclic and that $H$ can be decomposed
  as a non-trivial amalgamated free
  product~\cite[Section~3]{bdh}. Hence, $H^2_b(H;\R) \not\cong
  0$~\cite{grigorchuk,fujiwara}. Therefore, 
  \[ \bcd \Bigl( \bigoplus_\N H \Bigr) = \infty. 
  \]
  On the other hand, acyclicity of~$H$, the K\"unneth theorem, and the
  compatiblity of homology with colimits shows that $\bigoplus_\N H$
  is acyclic.
\end{exa}

However, so far, no examples of \emph{finitely generated} non-amenable
groups~$G$ seem to be known where $\bcd G$ can be computed explicitly.

\begin{question}
  What can be said about the bounded-cohomological dimension
  of~$(\bigoplus_\Z F_2) \rtimes \Z$, where $\Z$ acts on~$\bigoplus_\Z
  F_2$ by shifting the summands?
\end{question}

\subsection*{Organisation of this article}

In Section~\ref{sec:bcl1}, we briefly recall the definition of bounded
cohomology and $\ell^1$-homology of discrete groups, as well as some
basic properties and constructions. In Section~\ref{sec:large}, we
will give simple examples of groups with large bounded cohomology; in
particular, we will prove Proposition~\ref{prop:large}. Finally, in
Section~\ref{sec:small}, we will compute the bounded cohomology of
mitotic groups, which proves Theorem~\ref{thm:small}.

\pagebreak

\section{Bounded cohomology and $\ell^1$-homology}\label{sec:bcl1}

We briefly review the definitions and basic properties of bounded
cohomology and $\ell^1$-homology of (discrete) groups with constant
coefficients: 

\subsection{Bounded cohomology and $\ell^1$-homology}

Bounded cohomology and $\ell^1$-homology are normed refinements 
of classical group (co)homology: We will use the following concrete 
description:

\begin{defi}[$\ell^1$-norm, bounded cohomology, $\ell^1$-homology]
  Let $G$ be a group. We denote the standard chain complex
  by~$C_*(G;\R)$; more precisely, for $k \in \N$ we write~$C_k(G;\R)
  := \bigoplus_{g \in G^k} \R \cdot g$ and
  \begin{align*}
    \partial_k \colon C_k(G;\R) & \longrightarrow C_{k-1}(G;\R) \\ 
    G^k \ni (g_1,\dots, g_k) & \longmapsto 
    (g_2, \dots, g_k)\\ 
    & \phantom{\longmapsto} + \sum_{j=1}^{k-1} (-1)^j \cdot (g_1, \dots, g_j
    \cdot g_{j+1}, \dots, g_k)\\ 
    & \phantom{\longmapsto} + (-1)^k \cdot (g_1, \dots, g_{k-1}).
  \end{align*}
  We denote the $\ell^1$-norm on~$C_k(G;\R)$ associated with the 
  basis~$G^k$ by~$\|\cdot\|_1$. Notice that~$\| \partial_k\| \leq k+1$ 
  with respect to the $\ell^1$-norms.
  \begin{itemize}
    \item The completion of~$C_*(G;\R)$ with respect to the
      $\ell^1$-norm is denoted by~$C_*^{\ell^1}(G;\R)$, the
      \emph{$\ell^1$-chain complex of~$G$}. 
    \item The topological dual
      of~$C_*(G;\R)$ with respect to the $\ell^1$-norm is denoted
      by~$C_b^*(G;\R)$, the \emph{bounded cochain complex of~$G$}.
    \item The homology~$H_*^{\ell^1}(G;\R)$ of~$C_*^{\ell^1}(G;\R)$ is
      called \emph{$\ell^1$-homology of~$G$}. The reduced
      homology~$\rlhom * G$ (i.e., kernel modulo closure of the image
      of the boundary operator) of~$C_*^{\ell^1}(G;\R)$ is called
      \emph{reduced $\ell^1$-homology of~$G$}.
    \item The cohomology~$H^*_b(G;\R)$ of~$C_b^*(G;\R)$ is 
      \emph{bounded cohomology of~$G$}. The reduced cohomology~$\rbc * G$ 
      of~$C_b^*(G;\R)$ is called \emph{reduced bounded cohomology of~$G$}.
  \end{itemize}
\end{defi}

Clearly, all these constructions are functorial with respect to group
homomorphisms and the inclusion~$C^*_b(\args;\R) \hookrightarrow
C^*(\args;\R)$ induces a natural transformation between bounded
cohomology and ordinary group cohomology, the so-called
\emph{comparison map}.

The $\ell^1$-norm and its dual norm induce semi-norms on
$\ell^1$-homology and bounded cohomology, respectively. By definition,
these semi-norms are norms on reduced $\ell^1$-homology and reduced
bounded cohomology, which then consist of Banach spaces.

More background on (co)homology of normed (co)chain complexes and on
descriptions of bounded cohomology and $\ell^1$-homology in terms of
homological algebra can be found in the
literature~\cite{vbc,ivanov,monod,mm,loehl1,buehler}.

\subsection{Evaluation and duality}

Evaluation gives rise to a weak form of duality between bounded
cohomology and $\ell^1$-homology. If $G$ is a group and $k \in \N$,
then the evaluation map
\begin{align*}
  \langle \args, \args \rangle 
  \colon C_b^k(G;\R) \otimes_\R C_k(G;\R)  
  & \longrightarrow \R
  \\
  (f,c) & \longmapsto f(c)
\end{align*}
is compatible with the (co)boundary operators and it is continuous
with respect to the (dual) $\ell^1$-norm and hence induces a well-defined
natural \emph{Kronecker product}
\begin{align*}
  \langle \args, \args\rangle \colon
  \rbc k G \otimes_\R \rlhom k G 
  & \longrightarrow \R.
\end{align*}

\begin{prop}[weak duality principle~\protect{\cite{mm,loehl1}}]\label{prop:duality}
  Let $G$ be a group and let $k \in \N$. Then the map
  \[ \rbc k G \longrightarrow \bigl( \rlhom k G \bigr)'
  \]
  induced by the Kronecker product is surjective.
\end{prop}

\begin{prop}[bounded cohomology and $\ell^1$-acyclicity~\protect{\cite{mm}}]
  Let $G$ be a group. Then $\bcd G = 0$ if and only if $G$ is
  \emph{$\ell^1$-acyclic}, i.e., $H_k^{\ell^1}(G;\R) \cong 0$ for
  all~$k \in \N_{>0}$.
\end{prop}

\subsection{The cross-product in bounded cohomology and $\ell^1$-homology}\label{subsec:product}

The explicit descriptions of the (co)homological cross-products are
continuous with respect to the (dual) $\ell^1$-norm and lead to
well-defined cross-products in bounded cohomology and
$\ell^1$-homology:

For groups~$G$, $H$ the homological cross-product is induced from the 
maps
\begin{align*}
  \args\times\args \colon C_p(G;\R) \otimes_\R C_q(H;\R) 
  & \longrightarrow C_{p+q}(G \times H;\R) 
  \\
  (g_1, \dots, g_p) \otimes 
  (h_1, \dots, h_q)
  & \longmapsto \sum_{\sigma \in S_{p,q}} (-1)^{|\sigma|} \cdot 
  \bigl( (g_{\sigma_1(j)}, h_{\sigma_2(j)} \bigr)_{j \in \{1,\dots, p+q\}}.
\end{align*}
Here, $S_{p,q}$ is the set of all
$(p,q)$-shuffles~$\sigma=(\sigma_1,\sigma_2)$~\cite{dold}, and 
$|\sigma|$ denotes the sign of shuffles~$\sigma \in S_{p+q}$.

This cross-product is bounded in every degree with respect
to the norms induced from the $\ell^1$-norm. Because the
compatibility with the boundary operators carries over to the
completed chain complexes, we obtain a corresponding well-defined
natural cross-product on (reduced) $\ell^1$-homology.

Dually, for groups~$G$ and $H$ the cohomological cross-product is induced 
from the maps
\begin{align*}
  \args\times\args\colon
  C^p(G;\R) \otimes_\R C^q(H;\R) & \longrightarrow C^{p+q}(G \times H;\R) 
  \\
  \varphi \otimes \psi 
  & \longmapsto (-1)^{p \cdot q} \cdot 
                \bigl( ((g_1,h_1), \dots, (g_{p+q}, h_{p+q}))
  \\
  & \quad\quad
                \mapsto
                \varphi(g_1, \dots, g_p) \cdot \psi(h_{p+1}, \dots, h_{p+q})
                \bigr),
\end{align*}
as suggested by the Alexander-Whitney map. These maps preserve
boundedness and are continuous and thus induce a well-defined natural
cross-product on (reduced) bouneded cohomology.

\begin{defi}[cross-product on bounded cohomology/$\ell^1$-homology]
  Let $G$ and $H$ be groups and let $p,q \in \N$. Then the
  \emph{cross-product} on reduced $\ell^1$-homology and reduced
  bounded cohomology are defined via:
  \begin{align*}
    \args\times \args
    \rlhom p G \otimes_\R \rlhom q H 
    & \longrightarrow \rlhom{p+q} {G \times H}
    \\
    [c] \otimes [d] & \longmapsto [c\times d]
    \\
    \args\times\args \rbc p G \otimes_\R \rbc q H 
    & \longmapsto \rbc{p+q}{G\times H} 
    \\
    [f] \otimes [g] & \longmapsto [f \times g].
  \end{align*}
\end{defi}

As in the case of ordinary group (co)homology these cross-products 
are compatible in the following sense:

\begin{prop}[compatibility of cross-products]\label{prop:crosscomp}
  Let $G$ and $H$ be groups, let $p,q \in \N$ and let $\varphi \in
  \rbc p G$, $\psi \in\rbc q H$ as well as $\alpha \in
  \rlhom p G$, $\beta \in \rlhom q H$. Then
  \[ \langle \varphi \times \psi, \alpha \times \beta \rangle
     = (-1)^{p \cdot q} \cdot \langle \varphi,\alpha\rangle
       \cdot \langle \psi, \beta\rangle.
  \]
\end{prop}
\begin{proof}
  For classical group (co)homology this can be deduced from the
  above explicit descriptions of the cross-products on the (co)chain level
  and the fact that the Alexander-Whitney map~$A$
  satisfies~$A \circ (\args\times \args) \simeq \id$ on the (co)chain
  level (Lemma~\ref{lem:awez} below).

  Because this natural chain homotopy~$\Omega$ can be chosen to be
  bounded in each degree (Lemm~\ref{lem:awez}), the corresponding
  arguments carry over to the $\ell^1$-chain complex and the bounded
  cochain complex:

  Let $f \in C^p_b(G;\R)$, $g \in C^q_b(H;\R)$, $c \in
  C_p^{\ell^1}(G;\R)$, $d \in C_q^{\ell^1}(H;\R)$ be (co)cycles
  representing~$\varphi, \psi, \alpha, \beta$, respectively. Let
  $\overline C_*$ be the completion of~$C_*^{\ell^1}(G;\R) \otimes_\R
  C_*^{\ell^1}(H;\R)$ with respect to the norm induced by the
  $\ell^1$-norms. Then $A$ extends to a chain map~$\overline
  A \colon C_*^{\ell^1}(G\times H;\R) \longrightarrow \overline C_*$ 
  that is bounded in each degree, 
  and also $\Omega$ extends to~$\overline \Omega$ satisfying
  \[ \overline A \circ (\args \times \args) - \id
     = \partial \overline \Omega + \overline \Omega \circ \partial.
  \]
  Moreover, $f \otimes g$ also can be evaluated on elements
  of~$\overline C_{p+q}$ because $f$ and $g$ are bounded. Therefore,
  \begin{align*}
    (-1)^{p \cdot q} \cdot 
    (f \times g)(c \times d)
    & = (f \otimes g) 
        \bigl( A \circ (\args \times \args) (c \otimes d)\bigr)
        \\
    & = (f \otimes g) 
        \bigl( \overline A \circ (\args \times \args) (c \otimes d)\bigr)
        \\
    & = (f \otimes g) (c \otimes d) 
        \\
    &
        - (f \otimes g) \bigl( \partial \circ \overline \Omega (c \otimes d)\bigr)
        - (f \otimes g) \bigl( \overline \Omega \circ \partial (c \otimes d) \bigr)
        \\
    & = (f \otimes g) (c \otimes d)\\
    & = f(c) \cdot g(d),
  \end{align*}
  as desired.
\end{proof}

\begin{lem}\label{lem:awez}
  Let $G$ be a group. Then the cross-product 
  \[ \args \times \args \colon C_*(G;\R) \otimes_\R C_*(G;\R) 
     \longmapsto C_*(G\times G;\R)
  \]
  and the Alexander-Whitney map given by 
  \begin{align*}
    A \colon C_q(G \times G) & \longrightarrow \bigl(C_*(G) \otimes_\R C_*(G)\bigr)_q \\
    (G \times G)^q \ni
    \bigl((g_1,h_1), \dots, (g_q, h_q)\bigr) 
    & \longmapsto 
    \sum_{j=0}^q
    (g_1, \dots, g_j) \otimes (h_{j+1}, \dots, h_q)
  \end{align*}
  are natural chain maps that are mutually chain homotopy inverses of
  each other.  More precisely, there exist natural chain homotopies
  \begin{align*}
    \Xi \colon (\args\times\args) \circ A & \simeq \id
    \\
    \Omega \colon A \circ (\args\times\args) & \simeq \id
  \end{align*} 
  that are bounded in each degree (with respect to the norms induced
  from the respective \mbox{$\ell^1$-norms}), where the bounds in every degree~$q$ depend
  only on~$q$ and \emph{not} on the group~$G$.
\end{lem}
\begin{proof}
  This is a consequence of the classic proof via the acyclic model 
  theorem~\cite{dold}. 
\end{proof}

A more systematic study of acyclic models in the context of
$\ell^1$-homology was carried out by
Bouarich~\cite{bouarich}. Moreover, for sufficiently well-behaved
products the spectral sequence of Monod applies~\cite{monod}.

Furthermore, (reduced) bounded cohomology carries a natural ring
structure via the cup-product:

\begin{defi}[cup-product on bounded cohomology]
  Let $G$ be a group, and let $p, q \in \N$. Then the \emph{cup-product on~$\rbc * G$} 
  is given by
  \begin{align*}
    \args\cup\args \colon \rbc p G \otimes_\R \rbc q G & \longrightarrow \rbc {p+q} {G} \\
    \varphi \otimes \psi 
    & \longmapsto 
    \rbc {p+q}{\Delta_G} (\varphi \times \psi),
  \end{align*}
  where $\Delta_G \colon G \longrightarrow G \times G$ is the diagonal map.
\end{defi}

As in classical group cohomology, also the relation
\[ \varphi \times \psi 
   = \rbc{p}{p_G} (\varphi) \cup \rbc q {p_H} (\psi)
   \in \rbc {p+q} {G \times H}
\]
holds for all~$\varphi \in \rbc p G$, $\psi \in \rbc q H$, where
$p_G\colon G \times H \longrightarrow G$ and $p_H \colon G \times H
\longrightarrow H$ are the projections onto the factors.

\section{Groups with large bounded cohomology}\label{sec:large}

We will now construct groups with large bounded cohomology by taking
(free) products and exploiting the relation with $\ell^1$-homology. In
particular, we will prove Proposition~\ref{prop:large} and related
results.

\subsection{$\ell^1$-Betti numbers}\label{subsec:l1betti}
We introduce (reduced) $\ell^1$-Betti numbers of groups and
discuss their basic properties as well as their influence on bounded
cohomology.

\begin{defi}[$\ell^1$-Betti numbers]
  Let $G$ be a group and let $k \in \N$. Then the \emph{$k$-th 
  $\ell^1$-Betti number~$\lob k G$} is defined as the cardinality of 
  an $\R$-basis of~$\rlhom k G$; we also write~$\lob k G = \infty$ 
  if this cardinality is infinite.
\end{defi}

For example, $\ell^1$-Betti numbers satisfy the following simple 
inheritance properties:

\begin{prop}\label{prop:l1inherit} 
  Let $G$ be a group and let $k \in \N_{>0}$.
  \begin{enumerate}
    \item We have $\dim_\R H_b^k(G;\R) \geq \lob k G$. In particular:
      If $\lob k G \neq 0$, then we have~$\bcd G \geq k$.
    \item Conversely, if $H_b^2(G;\R) \not\cong 0$, then $\lob 2 G \neq 0$.
    \item If $H$ is a group that is a retract of~$G$, i.e., there are
      group homomorphisms $i \colon H \longrightarrow G$ and $r \colon
      G \longrightarrow H$ with~$r \circ i = \id_H$, then
      \[ \lob k G \geq \lob k H. 
      \]
    \item If $H$ is a group, then 
      \[ \lob k {G * H} \geq \lob k G + \lob k H. 
      \]
      In particular: If $\lob k G \neq 0$, then 
      $\lob[Big] k {\bigfree_\N G} = \infty.
      $
    \item If $G$ is countable and $\lob k G = \infty$, then
      \begin{align*} 
        \dim_\R H^{\ell^1}_k(G;\R) =  \dim_\R \rlhom k G = |\R|,
        \\
        \dim_\R H_b^k(G;\R) =  \dim_\R \rbc k G = |\R|.
      \end{align*}
  \end{enumerate}
\end{prop}
\begin{proof}
  The first part follows from Proposition~\ref{prop:duality}.  The
  second part follows from an observation of Matsumoto and
  Morita~\cite[Corollary~2.7 and Theorem~2.3]{mm}. 
  The third part is a direct
  consequence of functoriality of reduced $\ell^1$-homology. 

  The fourth part can be shown as follows: 
  From Proposition~\ref{prop:duality} we deduce that there exist
  families~$(\varphi_i)_{i \in I} \subset \rbc k G$, $(\alpha_i)_{i
    \in I} \subset \rlhom k G$ and $(\psi_j)_{j \in J} \subset \rbc k
  H$, $(\beta_j)_{j \in J} \subset \rlhom k H$ with~$|I| = \lob k G$
  and $|J| = \lob k H$ that satisfy
  \[ \langle \varphi_i, \alpha_{i'} \rangle = \delta_{i,i'}
     \quad\text{and}\quad
     \langle \psi_j, \beta_{j'} \rangle = \delta_{j,j'}
  \]
  for all~$i,i' \in I$ and all~$j, j' \in J$.
  Let $i_G \colon G \longrightarrow G * H$, $i_H \colon H \longrightarrow G *H$, 
  $p_G \colon G * H \longrightarrow G$, $p_H \colon G * H \longrightarrow H$ 
  be the canonical inclusions and projections associated with the free factors. 
  Then 
  \begin{align*}
    \bigl\langle 
    \rbc k {p_G}(\varphi_i)
    , \rlhom k {i_G} (\alpha_{i'})
    \bigr\rangle
    & = \langle \varphi_i, \alpha_{i,i'} \rangle
    = \delta_{i,i'},
    \\
    \bigl\langle
    \rbc k {p_G}(\varphi_i)
    , \rlhom k {i_H} (\beta_j)
    \bigr\rangle
    & = 
    \bigl\langle
    \varphi_i
    , \rlhom k 1 (\beta_j)
    \bigr\rangle
    = \langle \varphi_i, 0 \rangle
    = 0
  \end{align*}
  etc. for all~$i ,i' \in I$, $j,j' \in J$.
  Hence, $\lob k {G * H} \geq |I| + |J| = \lob k G + \lob k H$.

  We now prove the last part: By definition, $\rlhom k G$ and $\rbc k
  G$ are Banach spaces, and Banach spaces of infinite dimension have
  dimension at least~$|\R|$. On the other hand,
  countability of~$G$ implies that we have both $\dim_\R C^{\ell^1}_k(G;\R) \leq
  |\R|$ and $\dim_\R C_b^k(G;\R) \leq |\R|$. Hence,
  \begin{align*} 
    |\R| \leq \dim_\R \rlhom k G \leq \dim_\R H^{\ell^1}_k(G;\R) 
    \leq \dim_\R C^{\ell^1}_k(G;\R) \leq |\R|
    \\
    |\R| \leq \dim_\R \rbc k G \leq \dim_\R H^k_b(G;\R) 
    \leq \dim_\R C^k_b(G;\R) \leq |\R|. &
    \qedhere
  \end{align*}
\end{proof}

While it is not clear whether $\ell^1$-homology or bounded cohomology 
satisfy a simple K\"unneth theorem, we at least have the following weak version:

\begin{prop}\label{prop:l1kuenneth}
  Let $G$ and $H$ be groups and let $p,q \in \N$. Then
  \[ \lob {p+q} {G \times H} \geq \lob p G \cdot \lob q H. 
  \]
  In particular: If $\lob p G \neq 0$ and $\lob q H \neq 0$, then 
  $\bcd (G \times H) \geq p + q$.
\end{prop}
\begin{proof}
  From Proposition~\ref{prop:duality} we deduce that there exist
  families~$(\varphi_i)_{i \in I} \subset \rbc p G$, $(\alpha_i)_{i
    \in I} \subset \rlhom p G$ and $(\psi_j)_{j \in J} \subset \rbc k
  H$, $(\beta_j)_{j \in J} \subset \rlhom k H$ with~$|I| = \lob p G$
  and $|J| = \lob q H$ that satisfy
  \[ \langle \varphi_i, \alpha_{i'} \rangle = \delta_{i,i'}
     \quad\text{and}\quad
     \langle \psi_j, \beta_{j'} \rangle = \delta_{j,j'}
  \]
  for all~$i,i' \in I$ and all~$j, j' \in J$. Hence, the compatibility 
  of the cross-products (Proposition~\ref{prop:crosscomp}) yields 
  \[ (-1)^{p \cdot q} \cdot 
     \langle \varphi_i \times \psi_j, \alpha_{i'} \times \beta_{j'} \rangle
     = \langle \varphi_i, \alpha_{i'}\rangle 
       \cdot \langle \psi_j, \beta_{j'} \rangle 
     = \delta_{i,i'} \cdot \delta_{j,j'} 
     = \delta_{(i,j), (i',j')}
  \] 
  for all~$(i,j),(i',j') \in I \times J$; thus,~$\lob {p+q} {G \times
    H} \geq |I| \cdot |J| \geq \lob p G \cdot \lob q H$.

  The second part follows then with help of Proposition~\ref{prop:l1inherit}.
\end{proof}

\subsection{Examples}\label{subsec:examples}

The observations from Section~\ref{subsec:l1betti} are now applied to 
concrete examples: 

\begin{defi}[infinite chains of cup-products in bounded cohomology]
  Let $G$ be a group. Then $G$ \emph{admits infinite chains of
    cup-products in bounded cohomology} if there exists a
  sequence~$(\varphi_n)_{n \in \N} \subset \rbc * G$ of non-zero
  degree such that for all~$n \in \N$ we have
  \[ \varphi_0 \cup \dots \cup \varphi_{n-1} \neq 0 \in \rbc * G.
  \]
\end{defi}

\begin{prop}\label{prop:largegroups}
  Let $G_0$ be a group with~$\lob 3 {G_0} = \infty$ and for each~$n
  \in \N_{>0}$ let $G_n$ be a group with~$\lob 2 {G_n} = \infty$. Let
  $G$ be $\bigoplus_{n \in \N} G_n$ or~$\prod_{n \in \N} G_n$.  Then
  \[ \lob k G
     = 
     \begin{cases}
       1 & \text{if $k=0$}\\
       0 & \text{if $k=1$}\\
       \infty & \text{if $k\in \N_{\geq 2}$}
     \end{cases}
     \quad \text{and}\quad
     \dim_\R \rbc k G
     = 
     \begin{cases}
       1 & \text{if $k=0$}\\
       0 & \text{if $k=1$}\\
       \infty & \text{if $k\in \N_{\geq 2}$}
     \end{cases}
  \]
  for all~$k \in \N$, and thus $\bcd(G) = \infty$. 
  Moreover, $G$ admits infinite chains of cup-products in bounded
  cohomology and for all~$k \in \N_{\geq 4}$ there exist non-trivial
  classes in~$H^k_b(G;\R)$ that decompose as cup-products of classes
  in degree~$2$ and~$3$.
\end{prop}
\begin{proof}
  We only need to consider the case~$k \geq 2$. Every~$k \in \N_{\geq
    2}$ can be written in the form~$k = 2 \cdot r + 3 \cdot s$ with~$r
  \in \N$ and~$s \in \{0,1\}$. Because
  \[ G_1 \times \dots \times G_r 
     \quad\text{and}\quad
     G_0 \times G_1 \times \dots \times G_r
  \]
  are retracts of~$G$, the calculation of
  the dimensions follows from Proposition~\ref{prop:l1inherit} and
  Proposition~\ref{prop:l1kuenneth}. The assertion on the cup-products
  follows from the argumentation via iterated cross-products and the
  relation between the cohomological cross-product and the cup-product
  on bounded cohomology.
\end{proof}

\begin{rem}[exact cardinality]
  If $G_0, G_1, \dots$ are countable groups that satisfy the
  assumptions of Proposition~\ref{prop:largegroups} and $G :=
  \bigoplus_{n \in \N} G_n$, then
  \[ \dim_\R H^k_b(G;\R) = \dim_\R \rbc k G = |\R| 
  \]
  for all~$k \in \N_{\geq 2}$ by Proposition~\ref{prop:largegroups} and
  Proposition~\ref{prop:l1inherit}.

  Furthermore, by taking the infinite free product with the examples by
  Soma~\cite{somanonbanach}, we can also enforce that the difference 
  between reduced and non-reduced bounded cohomology is infinite-dimensional 
  in degree~$3, 5, 6, \dots$.
\end{rem}

\begin{prop}\label{prop:largeprod}
  For each~$n \in \N$ let $G_n$ be a group such that there exists a
  degree~$k_n \in \N_{>1}$ with~$\lob {k_n} {G_n} \neq 0$. Then
  \[ \bcd \Bigl( \bigoplus_{n \in \N} G_n\Bigr) = \infty
     \quad\text{and}\quad
     \bcd \Bigl( \prod_{n \in \N} G_n\Bigr) = \infty,
  \]
  and $\bigoplus_{n \in \N} G_n$ and $\prod_{n \in \N} G_n$ admit infinite 
  chains of cup-products in bounded cohomology.
\end{prop}
\begin{proof}
  Similarly to the proof of Proposition~\ref{prop:largegroups} this
  follows inductively from Proposition~\ref{prop:l1inherit} and
  Proposition~\ref{prop:l1kuenneth}.
\end{proof}

We can now also easily deduce a proof for Proposition~\ref{prop:large}:

\begin{proof}[Proof of Proposition~\ref{prop:large}]
  Because $H^2_b(G_n;\R) \not \cong 0$, we know that $\lob 2 G \neq 0$
  (Proposition~\ref{prop:l1inherit}). Therefore,
  Proposition~\ref{prop:largeprod} provides the desired conclusion.
\end{proof}

Some concrete groups with large $\ell^1$-Betti numbers or large
bounded-cohomological dimension are:

\begin{exa}
  Let $n \in \N_{\geq 2}$ and let $(M_k)_{k \in \N}$ be a sequence of
  oriented closed connected $n$-manifolds with positive simplicial
  volume, e.g., hyperbolic manifolds~\cite{vbc,thurston}. Then $\lob n
  {\pi_1(M_k)} \neq 0$~\cite{vbc}, and so
  Proposition~\ref{prop:l1inherit} shows that
  \[ \lob[Big] n {\bigfree_{k \in \N} \pi_1(M_k)} 
     = \infty.
  \]
  Moreover, Proposition~\ref{prop:largeprod} tells us that
  \[ \bcd \Bigl(\bigoplus_{k \in \N} \pi_1(M_k)\Bigr) = \infty
     \quad\text{and}\quad
     \bcd \Bigl(\prod_{k \in \N} \pi_1(M_k)\Bigr) = \infty.
  \]
\end{exa}

\begin{exa}
  It is well known that $\lob 2 {F_2} = \infty$~\cite{mitsumatsu}.  If
  $M$ is an oriented closed connected hyperbolic $3$-manifold, then
  $\lob 3 {\pi_1(M)} \neq 0$ (as in the previous example). Let $H :=
  \bigfree_\N \pi_1(M)$.  Hence, we have $\lob 3 {H} = \infty$ by
  Proposition~\ref{prop:l1inherit}. So,
  Proposition~\ref{prop:largegroups} allows us to compute the size of
  (reduced) bounded cohomology and (reduced) $\ell^1$-homology of~$H
  \times \bigoplus_\N F_2$ in all degrees. In particular,
  \[ \bcd \Bigl( H \times \bigoplus_\N F_2\Bigr) = \infty.\]
  Furthermore, by Proposition~\ref{prop:largeprod},
  \[ \bcd \Bigl( \bigoplus_\N F_2 \Bigr) = \infty
     \quad\text{and}\quad
     \bcd \Bigl( \prod_\N F_2 \Bigr) = \infty
     , 
  \]
  but because the exact structure of~$H^*_b(F_2;\R)$ is unknown, 
  it is currently out of reach to calculate the bounded cohomology 
  ring of~$\bigoplus_\N F_2$ or~$\prod_\N F_2$ completely.
\end{exa}

If we are not interested in having many non-trivial cup-products, then
we can also take large free products:

\begin{exa}
  For~$n \in \N$ let $G_n$ be a group with~$\lob {k_n} {G_n} \neq 0$ for
  some~$k_n \geq n$; e.g., we could take the fundamental group of an
  oriented closed connected hyperbolic $n$-manifold of dimension at
  least~$n$. Then Proposition~\ref{prop:l1inherit} shows that
  \[ \bcd \Bigl(\bigfree_{n \in \N} G_n \Bigr) = \infty. 
  \]
\end{exa}

\section{Groups with small bounded cohomology}\label{sec:small}

We will first recall the notion of mitotic groups and their basic
properties (Section~\ref{subsec:mitoticgroups}).  We will prove
Theorem~\ref{thm:small} in Section~\ref{subsec:proofsmall}, i.e., that
mitotic groups have bounded-cohomological dimension equal to~$0$. As a
preparation for this proof, we recall the uniform boundary condition
in Section~\ref{subsec:ubc}.

\subsection{Mitotic groups}\label{subsec:mitoticgroups}

We recall the notion of mitotic groups, due to Baumslag, Dyer,
Heller~\cite[Section~4]{bdh}.  Roughly speaking, a mitosis of a
group~$G$ is an ambient group that allows to divide~$G$ into two
copies of itself by means of conjugation
(Figure~\ref{fig:mitosis}). For group elements~$g, h$ we use the
conjugation notation
$g^h := h \cdot g \cdot h^{-1}. 
$

\begin{defi}[mitotic group]
  Let $G$ be a subgroup of a group~$M$. Then $M$ is a \emph{mitosis of~$G$} 
  if there exist~$s, d \in M$ with the following properties:
  \begin{enumerate}
    \item The group~$M$ is generated by~$G \cup \{s, d\}$.
    \item For all~$g \in G$ we have~$g^d = g \cdot g^s$.
    \item For all~$g,g' \in G$ we have~$[g', g^s] = 1$.
  \end{enumerate}
  We then also call the inclusion~$G \hookrightarrow M$ a mitosis 
  and the elements~$d,s$ as above are \emph{witnesses} for this mitosis. 
  A group~$M$ is \emph{mitotic}, if for every finitely generated
  subgroup~$G \subset M$ there exists a subgroup~$M' \subset M$ such
  that $G \subset M'$ is a mitosis of~$G$.
\end{defi}

\begin{figure}
  \begin{center}
    \begin{tikzpicture}[x=0.8cm,y=0.8cm]
      \draw (0,0) circle (0.5);
      \draw (0,0) node {$G$};
      \draw (-1,-2) circle (0.5);
      \draw (-1,-2) node {$G$};
      \draw (1,-2) circle (0.5);
      \draw (1,-2) node {$G$};
      \begin{scope}[shift={(0,-1.3)}]
        \draw (0,0) -- (0,0.5);
        \draw[->] (0,0) -- (210:0.5);
        \draw[->] (0,0) -- (330:0.5);
      \end{scope}
      \draw (-1,-1) node[anchor=east] {conjugation by~$d$}; 
      \draw (1,-1) node[anchor=west] {conjugation by~$s$}; 
    \end{tikzpicture}
  \end{center}
  \caption{A mitosis of a group, schematically; the original group~$G$ commutes with~$G^s$ inside~$M$ and $g$ commutes with~$g^d$ for all~$g \in G$.} 
  \label{fig:mitosis}
\end{figure}
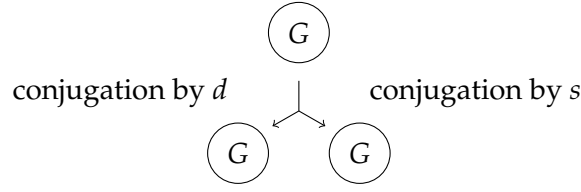

If $M$ is a mitosis of a group~$G$, witnessed by~$s, d \in M$, then
\begin{align*}
  \mu \colon G \times G & \longrightarrow M \\
  (g',g) & \longmapsto g' \cdot g^s
\end{align*}
is a well-defined group homomorphism; if $\Delta_G \colon G
\longrightarrow G \times G$ denotes the diagonal, then $\mu \circ
\Delta_G$ is nothing but conjugation with~$d$. Using the K\"unneth
theorem, the fact that conjugations act trivially on homology, and an 
induction argument, Baumslag, Dyer, Heller~\cite[Theorem~4.2]{bdh}
established that mitotic groups are acyclic:

\begin{thm}\label{thm:bdh}
  All mitotic groups are acyclic.
\end{thm}

In view of the universal coefficient theorem, we obtain that also
group cohomology with $\R$-coefficients is trivial for mitotic groups.

\subsection{The uniform boundary condition}\label{subsec:ubc}

We we will now review the uniform boundary condition, as studied by
Matsumoto and Morita~\cite{mm}. 

\begin{defi}[uniform boundary condition]
  Let $q \in \N$ and let $\kappa \in \R_{>0}$. A group~$G$ satisfies 
  the \emph{$(q,\kappa)$-uniform boundary condition ($(q,\kappa)$-UBC)} 
  if the following holds: for all~$z \in \im \partial_{q+1} \subset C_q(G;\R)$ 
  there is a chain~$c \in C_{q+1}(G;\R)$ with
  \[ \partial_{q+1} (c) = z 
     \quad\text{and}\quad
     \| c\|_1 \leq \kappa \cdot \|z\|_1.
  \]
  A group~$G$ satisfies~\emph{$q$-UBC} if there exists a~$\kappa \in \R_{>0}$ 
  such that $G$ satisfies~$(q,\kappa)$-UBC.
\end{defi}

For example, the uniform boundary condition allows to upgrade acyclicity 
of a group to vanishing of bounded cohomology~\cite[Theorem~2.8]{mm}:

\begin{thm}\label{thm:ubc}
  Let $G$ be a group and let $q \in \N$. Then the following are 
  equivalent:
  \begin{enumerate}
    \item The group~$G$ satisfies~$q$-UBC.
    \item The comparison map~$H_b^{q+1}(G;\R) \longrightarrow H^{q+1}(G;\R)$ 
      is injective.
  \end{enumerate}
  In particular: If $G$ is acyclic and $G$ satisfies $q$-UBC, then
  $H_b^{q+1}(G;\R) \cong 0$.
\end{thm}

More geometrically, the uniform boundary condition also has
applications in the context of simplicial volume of non-compact
manifolds~\cite{loehsauer}.

We introduce the following version of the uniform 
boundary condition:

\begin{defi}[uniform boundary condition]
  Let $q \in \N$, $\kappa \in \R_{>0}$. A group homomorphism~$\varphi
  \colon H \longrightarrow K$ satisfies the \emph{$(q,\kappa)$-uniform
    boundary condition ($(q,\kappa)$-UBC)} if there exists a linear
  map
  \[ S \colon \partial_{q+1}\bigl(C_{q+1}(H;\R)\bigr)
  \longrightarrow C_{q+1}(K;\R)\]  
  with
  \[ \partial_{q+ 1} \circ S = C_q(\varphi;\R)|_{\im \partial_{q+1}} 
     \quad\text{and}\quad
     \|S\| \leq \kappa.
  \]
  Here, $\|S\|$ denotes the norm of~$S$ with respect to the
  restricition of the \mbox{$\ell^1$-norm} to~$\partial_{q+1}(C_{q+1}(H;\R))$ and the
  $\ell^1$-norm on~$C_{q+1}(K;\R)$.
\end{defi}

Clearly, every group homomorphism satisfies~$(0,0)$-UBC.

\subsection{Bounded cohomology of mitotic groups}\label{subsec:proofsmall}

We will now prove Theorem~\ref{thm:small}, i.e., that mitotic groups have 
trivial bounded cohomology. The proofs of Baumslag, Dyer, Heller and 
the normed refinement of Matsumoto and Morita of Mather's argument for~$\HK n$ 
serve as a blueprint. 

In view of Theorem~\ref{thm:ubc} and Theorem~\ref{thm:bdh} we only
need to show that mitotic groups satisfy the uniform boundary
condition in each positive degree. To this end, we first prove that
mitoses allow to increase the degree in which the uniform boundary
condition is satisfied. More precisely, following the arguments of
Matsumoto and Morita~\cite{mm} step by step, one obtains the following
(a detailed proof is given in Appendix~\ref{appx:proof}):

\begin{prop}\label{prop:ubcmitosis}
  Let $q \in \N$, $\kappa \in \R_{>0}$. Then there is a
  constant~$c_{q,\kappa} \in \R_{>0}$ such that: let
  \[ \xymatrix{%
       H \ar[r]^\varphi 
       & H' \ar[r]^{\varphi'}
       & K \ar[r]^\psi
       & G \ar[r]^i 
       & M
    }
  \]
  be a chain of group homomorphisms with the following properties:
  \begin{itemize}
    \item The homomorphism~$i \colon G \hookrightarrow M$ is a mitosis.
    \item For all~$k \in \{1,\dots, q-1\}$ we have~$H_k(\varphi';\R) = 0$. 
    \item For all~$k \in \{0,\dots, q-1\}$ the group homomorphisms~$\varphi
      \colon H \longrightarrow H'$ and $\psi \colon K \longrightarrow
      G$ satisfy~$(k,\kappa)$-UBC.
  \end{itemize}
  Then for all~$k \in \{1,\dots,q\}$ we obtain
  \[ H_k(i \circ \psi \circ \circ \varphi' \circ \varphi;\R) = 0 
  \]
  and the composition $i \circ \psi \circ \varphi' \circ \varphi$ satisfies 
  $(k, c_{q,\kappa})$-UBC for all~$k \in \{0,\dots,q\}$.
\end{prop}

We can then easily complete the proof of Theorem~\ref{thm:small} 
by induction:

\begin{proof}[Proof of Theorem~\ref{thm:small}]
  Let $M$ be a mitotic group, let $q\in \N_{>0}$, and let $z \in
  C_q(M;\R)$ be a boundary, say~$z = \partial_{q+1}(c)$ for some~$c
  \in C_{q+1}(M;\R)$. Because $z$ and $c$ are finite linear
  combinations of tuples of~$M$, there exists a finitely generated
  subgroup~$G_0$ such that $z \in C_q(G_0;\R)$ and $c \in
  C_{q+1}(G_0;\R)$; i.e., $z$ is a boundary in~$G_0$.

  As $M$ is mitotic, we can extend~$G_0$ to a sequence~$G_0 \subset
  G_1 \subset G_2 \subset \dots M$ of finitely generated subgroups
  of~$M$ such that each step~$G_j \hookrightarrow G_{j+1}$ is a
  mitosis. 
  We now proceed by induction over~$q$: If $q = 1$, then 
  the sequence
  \[ \xymatrix{%
      G_0 \ar[r]^-{i_0}
      & G_1 \ar[r]^-{i_1}
      & G_2 \ar[r]^-{i_2}
      & G_3 \ar[r]^-{i_3}
      & G_4
    }
  \]
  of mitoses satisfies the assumptions of
  Proposition~\ref{prop:ubcmitosis} in degree~$1$, and hence the
  composition~$i_3 \circ i_2 \circ i_1 \circ i_0$ satisfies $(1,c_{1,0})$-UBC.

  For the induction step, let $q \in \N_{>1}$, let 
  \[ n_q := \sum_{j=0}^q 3^q 
  \]
  and suppose that compositions of $n_{q-1}$~mitoses in~$M$
  satisfy~$(q-1,\kappa_{q-1})$-UBC, where $\kappa_{q-1}$ depends only
  on~$q$, but not on the groups involved.

  Then the chain 
  \[ \xymatrix{%
       G_0 \ar[r]
       & G_{n_{q-1}} \ar[r]
       & G_{2 \cdot n_{q-1}} \ar[r]
       & G_{3 \cdot n_{q-1} = n_q -1} \ar[r]^-{i_{n_q - 1}}
       & G_{n_q}
    }
  \]
  of inclusions satisfies the assumptions of
  Proposition~\ref{prop:ubcmitosis} in degree~$q$, and hence the
  inclusion~$G_0 \hookrightarrow G_{n_q}$ satisfies~$(q,c_{q,
    \kappa_{q-1}})$-UBC where $c_{q,\kappa_{q-1}}$ depends only
  on~$q$, but \emph{not} on~$z$ or the chain~$G_0 \subset G_1 \subset \dots$.

  In particular, there is a chain~$c' \in C_{q+1}(M;\R)$ with
  \[ \partial_{q+1} c' = z \in C_{q}(M;\R) 
     \quad\text{and}\quad
     \|c'\|_1 \leq \kappa_q \cdot \|z\|_1.
  \]
  Hence, $M$ satisfies $(q,\kappa_q)$-UBC. Because $M$ is acyclic by
  Theorem~\ref{thm:bdh}, we obtain~$H^q_b(M;\R) \cong 0$ from
  Theorem~\ref{thm:ubc}.
\end{proof}


\appendix
\section{Detailed proof of Proposition~\ref{prop:ubcmitosis}}\label{appx:proof}

For the convenience of the reader, we present a detailed proof of
Proposition~\ref{prop:ubcmitosis}, following the arguments of Matsumoto
and Morita~\cite{mm}:

\begin{proof}[Proof of Proposition~\ref{prop:ubcmitosis}]
  It suffices to prove the claims in degree~$q$. We abbreviate~$f :=
  \psi \circ \varphi' \circ \varphi$.  The fact that $H_q(i \circ
  f;\R) = 0$ was proved by Baumslag, Dyer,
  Heller~\cite[Proposition~4.1]{bdh}. However, in order to make the
  normed refinement more transparent, we repeat the argument:

  Let $d, s \in M$ be witnesses for the mitosis~$i \colon G \hookrightarrow
  M$. Then 
  \begin{align*}
    \mu \colon  G \times G & \longrightarrow M \\
    (g',g) & \longmapsto g' \cdot g^s
  \end{align*}
  is a group homomorphism. Denoting the diagonal maps by~$\Delta_H$, $\Delta_G$ 
  and the conjugations on~$M$ by~$\gamma_d = \args^d$, $\gamma_s = \args^s$, we obtain 
  \[ \gamma_d \circ i \circ f = \mu \circ \Delta_G \circ f 
     = \mu \circ (f \times f) \circ \Delta_H.
  \]
  On the other hand, the K\"unneth theorem 
  (and its naturality) and the homological assumption on~$H_*(f;\R)$ shows 
  that the diagram
  \[ \xymatrix{%
       H_q(H \times H;\R) 
       \ar[d]_{H_q(f \times f;\R)}
       \ar[r]^-{\txt{\raisebox{1em}{\scriptsize$H_q(p_1;\R) \oplus H_q(p_2;\R)$}}}
       & H_q(H;\R) \oplus H_q(H;\R)
       \ar[d]^{H_q(f;\R) \oplus H_q(f;\R)}
       \\
       H_q(G \times G ;\R)
       \ar[d]_{H_q(\mu;\R)}
       & H_q(G;\R) \oplus H_q(G;\R)
       \ar[d]^{H_q(i;\R) \oplus H_q(\gamma_s \circ i;\R)}
       \ar[l]^-{\txt{\raisebox{-1em}{\scriptsize$H_q(i_1;\R) + H_q(i_2;\R)$}}}
       \\
       H_q(M;\R)
       & H_q(M;\R) \oplus H_q(M;\R) \ar[l]^-{\id + \id}
  }
  \]
  is commutative; here, $i_1, i_2, p_1, p_2$ denote the corresponding 
  inclusions and projections of the factors. Hence, we obtain
  \begin{align*}
    H_q(\gamma_d;\R) \circ H_q(i \circ f) 
    & = H_q(\mu;\R) \circ H_q(f \times f;\R) \circ H_q(\Delta_H;\R)
    \\
    & = H_q(i \circ f;\R) + H_q(\gamma_s;\R) \circ H_q(i \circ f;\R).
  \end{align*}
  As conjugations act trivially on homology, $H_q(\gamma_d;\R) = \id 
  = H_q(\gamma_s;\R)$, and so $H_q(i \circ f ;\R) = 0$.

  We will now refine this argument and prove that $i\circ f$ satisfies
  a strong uniform boundary condition in degree~$q$:

  Let $S_0, \dots, S_{q-1}$ and $T_0, \dots, T_{q-1}$ be sections that
  witness that $\varphi$ and $\psi$ satisfy $(0,\kappa)$-UBC, \dots,
  $(q-1,\kappa)$-UBC; for simplicity, we omit the indices and denote
  all these maps by~$S$ or~$T$ respectively. Let $z \in
  B_q(H) := \partial_{q+1}(C_{q+1}(H;\R))$. We construct an explicit
  $\partial_{q+1}$-primitive for~$C_q(i \circ f;\R)$ in two steps: We
  first deal with the K\"unneth argument, and then we will take care
  of the conjugations.

  \emph{Normed refinement of the K\"unneth argument.}
  We first study the intermediate degree part of~$z$, viewed
  in~$C_*(H;\R) \otimes_\R C_*(H;\R)$, i.e., the chain
  \[ D(z) := A \circ \Delta (z) - z \otimes 1 - 1 \otimes z,  
  \]
  where, $A \colon C_*(H \times H;\R) \longrightarrow C_*(H;\R)
  \otimes_\R C_*(H;\R)$ is the Alexander-Whit\-ney map
  (Lemma~\ref{lem:awez}) and $\Delta
  := C_*(\Delta_H;\R)$.
  Moreover, we write $\varphi_* := C_*(\varphi;\R)$ etc.  

  Similar to Matsumoto and Morita~\cite[p.~544]{mm} we define the map
  \begin{align*} 
    E := & \ (\psi_* \otimes_\R \psi_*) \circ (\varphi'_* \otimes_\R \varphi'_*) 
             \circ (S \otimes_\R S) \circ (\id \otimes_\R \partial)
    \\
    + & \ \bigl(T \otimes_\R (\psi_* - T \circ \partial)\bigr) 
    \circ (\varphi'_* \otimes_\R \varphi'_*) 
    \circ \bigl(\varphi_* \otimes_\R \varphi_* 
          - \partial \circ (S \otimes_\R S) \circ (\id \otimes_\R \partial) \bigr)
    \\
    \colon & \ B_q \longrightarrow \bigl(C_*(G;\R) \otimes_\R C_*(G;\R)\bigr)_{q+1}
  \end{align*}
  on
  $ B_q := \im \bigl(\text{$\partial_{q+1, C_*(H;\R) \otimes_\R C_*(H;\R)}$}\bigr) 
            \cap \bigoplus_{j = 1}^{q-1} C_j(H;\R) \otimes_\R C_{q-j}(H;\R). 
  $

  \begin{lem}[explicit primitives for~$D(z)$]\label{lem:E}
  This map~$E$ has the following properties:
  \begin{enumerate}
    \item The map~$E$ is well-defined.
    \item We have~$D(z) \in B_q$ and the map~$E$ produces explicit
      primitives, i.e., 
      \[ (f_* \otimes_\R f_*) D(z) = \partial_{q+1} E \bigl( D(z)\bigr).
      \]
    \item Moreover,  
      \[ \| E \| \leq
         \kappa + 2 \cdot (q + 1) \cdot \kappa^2 
         \cdot \bigl(1+ (q+1) \cdot \kappa + (q+1)^2 \cdot \kappa^2\bigr)
      \]
      with respect to the norms induced by the respective $\ell^1$-norms. Notice 
      that this bound does only depend on~$q$ and~$\kappa$, but \emph{not} on 
      the groups or homomorphisms that are involved.
  \end{enumerate}
  \end{lem}
 
  The proof of this lemma is given below. We now continue with the
  proof of Proposition~\ref{prop:ubcmitosis}: In view of the
  naturality of the cross-product map~$B \colon C_*(\args;\R)
  \otimes_\R C_*(\args;\R) \longrightarrow C_*(\args \times \args;\R)$
  and Lemma~\ref{lem:awez} we obtain 
  \begin{align*}
    (f\times f)_* \circ \Delta (z) 
    & = (f \times f)_* \circ B \circ A \circ \Delta (z) 
      + (f \times f)_* (\partial \circ \Xi + \Xi \circ \partial) \circ \Delta(z)
      \\
    & = B \circ (f_* \otimes_\R f_*) \circ A \circ \Delta(z) + (f \times f)_* \circ \partial \circ \Xi \circ \Delta(z).  
  \end{align*}
  The construction of~$D(z)$ and the explicit primitives from Lemma~\ref{lem:E} now 
  lead to
  \begin{align*}
    (f\times f)_* \circ \Delta (z) 
    & = (f \times f)_* \circ B (z \otimes 1) + (f \times f)_* \circ B (1 \otimes z) 
      + \partial E' (z),
  \end{align*}
  where
  \[ E' := B \circ E \circ D + (f \times f)_* \circ \Xi \circ \Delta; 
  \]
  notice that $E'$ is bounded and that $\|E'\|$ admits a bound that
  only depends on~$q$ and~$\kappa$, but not on the specific groups or
  homomorphisms.
  By definition of the cross-product, we have~$B(z \otimes 1) =
  j_{1*}(z)$ and $B(1 \otimes z) = j_{2*}(z)$, where $j_1, j_2 \colon
  H \longrightarrow H \times H$ are the inclusions of the
  factors. Therefore,
  \begin{align*} 
    (f\times f)_* \circ \Delta (z) 
     & = (f \times f)_* \circ j_{1*} (z) + (f \times f)_* \circ j_{2*} (z) + \partial \circ E'(z).
    \\
     & = i_{1*} \circ f_* (z) + i_{2*} \circ f_* (z) + \partial \circ E'(z).
  \end{align*}

  \emph{Normed refinement of the conjugation argument.}
  Applying~$\mu_*$ to this equation and using the chain
  homotopy~$\Theta$ from Lemma~\ref{lem:conj} below associated with
  the conjugation by~$k := s \cdot d^{-1}$ on~$M$ leads then to 
  \begin{align*}
    (i \circ f)_* (z) 
    & = \bigl(\mu \circ (f \times f) \circ j_1\bigr)_* (z)
    \\
    & = \bigl(\mu \circ (f \times f) \circ \Delta_H \bigr)_* (z)
    - \bigl(\mu \circ (f\times f) \circ j_2\bigr)_*(z) 
    - \mu_* \circ \partial \circ E'(z)
    \\
    & =
    \gamma_{d*} \circ (i \circ f)_* (z) - \gamma_{s*} \circ (i \circ f)_*(z) 
    - \partial \circ \mu_* \circ E'(z)
    \\
    & = 
    \gamma_{d*} \circ (i \circ f)_* (z) - \gamma_{k*} \circ \gamma_{d*} \circ (i \circ f)_*(z) 
    - \partial \circ \mu_* \circ E'(z)
    \\
    & = (\partial \circ \Theta + \Theta \circ \partial) \circ (i \circ f)_*(z) 
    - \partial \circ \mu_* \circ E'(z)
    \\
    & = \partial \bigl( \Theta \circ (i \circ f)_* (z) - \mu_* \circ E' (z)\bigr).
  \end{align*}
  Because $\| \Theta \circ (i \circ f)_* - \mu_* \circ E'\|$ admits a
  bound~$c_{q,\kappa}$ on~$B_q(H)$ that only depends on~$q$
  and~$\kappa$ (as the same holds for~$E'$ and~$\Theta$) we see that
  $i \circ f$ satisfies $(q, c_{q,\kappa})$-UBC, as desired.
\end{proof}

\begin{proof}[Proof of Lemma~\ref{lem:E}]
  We mainly follow the corresponding arguments by Matsumoto and Morita. 
  
  \emph{Ad~1.} Showing that $E$ is well-defined is the most delicate
  point of the whole proof of Theorem~\ref{thm:small}. Let $x \in
  B_q$. Because $x$ is a boundary, a straightforward calculation shows 
  that
  \[ (\id \otimes_\R \partial) (x) \in \bigoplus_{j=1}^{q-1} B_j(H) \otimes_\R B_{q-1-j}(H); 
  \]
  here, one should also note that $B_0(H) = 0$ by definition of the
  chain complex~$C_*(H;\R)$. In particular, $(S \otimes_\R S)$ indeed 
  can be applied to~$(\id \otimes_\R \partial) (x)$. This takes care 
  of the first summand of~$E$ and the last part of the second summand 
  of~$E$. 

  For the remaining terms, we consider the element
  \[ U(x) := \bigl(
              \varphi_* \otimes \varphi_* 
              - \partial \circ (S\otimes_\R S) \circ (\id \otimes_\R \partial)
            \bigr)
            (x). 
  \]
  Using the fact that the maps of type~$S$ are sections of~$\varphi_*$
  on boundaries, one readily computes
  $(\id \otimes_\R \partial) \circ U(x) = 0.
  $ 
  Exactness of the tensor product over~$\R$ then implies that
  \[ U(x) \in  \bigoplus_{j=1}^{q-1} C_j(H;\R) \otimes_\R Z_{q-j}(H), 
  \]
  where $Z_*(H)$ denotes the cycles in~$C_*(H;\R)$. On the other hand,
  we clearly also have~$\partial U(x) = 0$, and so $(\partial
  \otimes_\R \id) \circ U(x) = 0$ and (again by exactness of the
  tensor product over~$\R$) it follows that
  \[ U(x) \in \bigoplus_{j=1}^{q-1} Z_j(H) \otimes_\R Z_{q-j}(H).
  \]
  By assumption, $H_k(\varphi';\R) = 0$ for all~$k \in \{1,\dots,q-1\}$; thus, 
  \[ (\varphi'_* \otimes_\R \varphi'_*) \circ U(x) 
     \in \bigoplus_{j=1}^{q-1} B_j(H) \otimes_\R B_{q-j}(H).
  \]
  In particular, we indeed can apply the maps of type~$T$ to all components
  of~$(\varphi'_* \otimes_\R \varphi'_*) U(x)$ and of~$(\id \otimes_\R
  \partial) \circ (\varphi'_* \otimes_\R \varphi'_*) \circ U(x)$. Therefore, $E$ 
  is well-defined.

  \emph{Ad~2.} Because $z$ is a boundary, a straightforward calculation 
  shows that also~$D(z)$ is a boundary. Moreover, by construction of~$D(z)$, 
  all summands of~$D(z)$ are of intermediate degree. Hence, $D(z) \in B_q$, 
  and so $E$ can indeed be applied to~$D(z)$. 

  Because $(\id \otimes_\R \partial) \circ U(D(z)) = 0$, a calculation
  shows that
  \begin{align*}
    \partial_{q+1} E\bigl(D(z)\bigr)
    & = (\psi_* \otimes_\R \psi_*) \circ (\varphi'_* \otimes_\R \varphi'_*) 
        \circ (\varphi_* \otimes_\R \varphi_*) \bigl(D(z)\bigr)\\
    & =  (f_* \otimes_\R f_*) D(z).
  \end{align*}

  \emph{Ad~3.} The bound on~$\|E\|$ follows directly from the explicit
  definition of~$E$ and corresponding bounds on the building blocks
  of~$E$: Chain maps induced by group homomorphisms have norm~$1$, the
  maps of type~$S$ and~$T$ have norms bounded by~$\kappa$
  (by assumption), and the boundary operator on~$C_q(\args;\R)$ has
  norm bounded by~$q + 1$.
\end{proof}

\begin{lem}\label{lem:conj}
  Let $G$ be a group and let $k \in G$. Then
  \begin{align*}
    \Theta_q \colon 
    C_q(G;\R) & \longrightarrow C_{q+1}(G;\R) \\
    G^q \ni (g_1, \dots, g_q) 
    & \longmapsto 
    \sum_{j=1}^{q+1} (-1)^j \cdot 
    (g_1, \dots, g_{j-1}, k, k^{-1} \cdot g_j \cdot k, \dots, k^{-1} \cdot g_q \cdot k)
  \end{align*}
  defines a chain homotopy between the identity
  and~$C_*(\gamma_k;\R)$, where $\gamma_k$ denotes the conjugation
  on~$G$ by~$k$. Moreover, for all~$q \in \N$ we have
  \[ \| \Theta_q \| \leq q+1. 
  \]
\end{lem}
\begin{proof}
  This is a straightforward computation.
\end{proof}



\medskip
\vfill

\noindent
\emph{Clara L\"oh}\\[.5em]
  {\small
  \begin{tabular}{@{\qquad}l}
    Fakult\"at f\"ur Mathematik\\
    Universit\"at Regensburg\\
    93040 Regensburg\\
    Germany\\
    \textsf{clara.loeh@mathematik.uni-regensburg.de}\\
    \textsf{http://www.mathematik.uni-regensburg.de/loeh}
  \end{tabular}}

\begin{thebibliography}{100}

  \bibitem{bdh}
    G.~Baumslag, E.~Dyer, A.~Heller.
    The topology of discrete groups, 
    \emph{J.~Pure Appl.\ Algebra}, 16(1), pp.~1--47, 1980.

  \bibitem{bouarich} A.~Bouarich. Th\'eor\`emes de Zilber-Eilemberg et
    de Brown en homologie~{$\ell^1$}, \emph{Proyecciones}, 23(2), 
    pp.~151--186, 2004.

  \bibitem{buehler}
    T.~B\"uhler. 
    \emph{On the algebraic foundations of bounded cohomology}, 
    \emph{Mem.\ Amer.\ Math.\ Soc.}, 214 (1006), 2011.


  \bibitem{dold}
    A.~Dold.
    Lectures on Algebraic Topology, 
    reprint of the 1972 edition, \emph{Classics in Mathematics}, Springer, 1995.

  \bibitem{fujiwara}
    K.~Fujiwara.
    The second bounded cohomology of an amalgamated free product of groups,
    \emph{Trans.\ Amer.\ Math.\ Soc.}, 352(3), pp.~1113--1129, 2000.

  \bibitem{grigorchuk}
    R.I.~Grigorchuk.
    Some results on bounded cohomology, 
    \emph{Combinatorial and geometric group theory (Edinburgh, 1993)}, 
    pp.~111--163, \emph{London Math.\ Soc.\ Lecture Note Ser.}, 204, 
    Cambridge University Press, 1995. 

  \bibitem{grigorchuk2}
    R.I.~Grigorchuk.
    Bounded cohomology of group constructions,
    \emph{Mat.\ Zametki}, 59(4), pp.~546--550, 1996; 
    translation in \emph{Math.\ Notes}, 59(3--4), pp.~392--394, 1996.

  \bibitem{vbc}
    M.~Gromov.
    Volume and bounded cohomology.
    \emph{Inst. Hautes \'Etudes Sci. Publ. Math.}, 56, pp.~5--99, 1983.

  \bibitem{ho}
    T.~Hartnick, A.~Ott.
    Bounded cohomology via partial differential equations,~I,
    preprint, arXiv:1310.4806 [math.GR], 2013.

   \bibitem{inoueyano} H.~Inoue, K. Yano. The Gromov invariant
     of negatively curved manifolds. \emph{Topology}, 21(1), 
     pp.~83--89, 1981.

  \bibitem{ivanov} N.V.~Ivanov. Foundations of the theory of
    bounded cohomology. \emph{J.~Soviet Math.}, 37,
    pp.~1090--1114, 1987.

  \bibitem{loehl1}
    C.~L\"oh.
    Isomorphisms in $\ell^1$-homology,  
    \emph{M\"unster J.\ of Math.}, 1, pp.~237--266, 2008. 

  \bibitem{loehsauer}
    C.~L\"oh, R.~Sauer.
    Simplicial volume of Hilbert modular varieties,
    \emph{Comment.\ Math.\ Helv.}, 84, pp.~457--470, 2009. 

  \bibitem{loeh} C.~L\"oh.
    \emph{Simplicial Volume}. 
    Bull. Man. Atl., pp. 7--18, 2011.

  \bibitem{mather}
    J.N.~Mather.
    The vanishing of the homology of certain groups of homeomorphisms, 
    \emph{Topology}, 10, pp.~297--298, 1971. 

  \bibitem{mm}
    S.~Matsumoto, S.~Morita.
    Bounded cohomology of certain groups of homeomorphisms, 
    \emph{Proc.\ Amer.\ Math.\ Soc.}, 94(3), pp.~539--544, 1985. 

  \bibitem{mineyev}
    I.~Mineyev. 
    Straightening and bounded cohomology of hyperbolic groups, 
    \emph{Geom.\ Funct.\ Anal.}, 11, pp.~807--839, 2001.

  \bibitem{mitsumatsu}
    Y.~Mitsumatsu.
    Bounded cohomology and $l^1$-homology of surfaces,
    \emph{Topology}, 23(4), 465--471, 1984. 

   \bibitem{monod} N.~Monod. \emph{Continuous Bounded Cohomology
     of Locally Compact Groups}. Volume~1758 of \emph{Lecture
     Notes in Mathematics}, Springer, 2001.

  \bibitem{monodicm}
    N.~Monod.
    An invitation to bounded cohomology,
    \emph{International Congress of Mathematicians}, Vol.~II, pp.~1183--1211, 
    Eur.\ Math.\ Soc., 2006. 

  \bibitem{somasurf}
    T.~Soma.
    Bounded cohomology of closed surfaces, 
    \emph{Topology}, 36(6), pp.~1221--1246, 1997.

   \bibitem{somanonbanach} 
     T.~Soma. Existence of non-Banach bounded cohomology. 
     \emph{Topology}, 37(1), pp.~179--193, 1998.

  \bibitem{thurston}
    W.~P.~Thurston.
    \emph{The geometry and topology of $3$-manifolds},
    mimeographed notes, 1979.          

  \bibitem{yoshida}
   T.~Yoshida.
   On $3$-dimensional bounded cohomology of surfaces. 
   \emph{Homotopy theory and related topics (Kyoto, 1984)}, pp.~173--176, 
   \emph{Adv.\ Stud.\ Pure Math.}, 9, North-Holland, 1987. 

\end{thebibliography}
\end{document}